\documentclass[10pt]{amsart}

\usepackage{amsmath}
\usepackage{amssymb}
\newtheorem{proposition}{Proposition}[section]
\newtheorem{theorem}[proposition]{Theorem}
\newtheorem{corollary}[proposition]{Corollary}

\theoremstyle{definition}
\newtheorem{definition}[proposition]{Definition}

\newtheorem{remark}[proposition]{Remark}

\numberwithin{equation}{section}
\newcommand{\R}{\mathbb R}

\newcommand{\E}{\mathcal{E}}

\renewcommand{\leq}{\leqslant}

\def\Dx{\Delta_x}
\def\Cal{\mathcal}
\def\({\left(}
\def\){\right)}
\def\Nx{\nabla_x}
\def\Dt{\partial_t}

\def\eb{\varepsilon}

\begin{document}

\title[Hyperbolic Cahn-Hilliard-Oono equation]{Global well-posedness and attractors for the hyperbolic Cahn-Hilliard-Oono equation in the whole space }
\author[A. Savostianov and S. Zelik]{Anton Savostianov${}^1$ and Sergey Zelik${}^1$}

\begin{abstract} We prove the global well-posedness of the so-called hyperbolic relaxation of the Cahn-Hilliard-Oono equation in the whole space $\R^3$ with the non-linearity of the sub-quintic growth rate. Moreover, the dissipativity and the existence of a smooth global attractor in the naturally defined energy space is also verified. The result is crucially based on the Strichartz estimates for the linear Scr\"odinger equation in $\R^3$.
\end{abstract}

\subjclass[2000]{35B40, 35B45}
\keywords{Hyperbolic Cahn-Hilliard-Oono equation, attractors, unbounded domains, Strichartz estimates}
%\thanks{
%This work is partially supported by the Russian Ministry of Education and Science (contract no.
%8502).}
\address{${}^1$ University of Surrey, Department of Mathematics, \newline
Guildford, GU2 7XH, United Kingdom.}
\thanks{The authors would like to thank Nikolas Burq, Giulio Schimperna and Fabrice Planchon for many stimulating discussions}
\email{s.zelik@surrey.ac.uk}
\email{a.savostianov@surrey.ac.uk}

\maketitle
\tableofcontents

\section{Introduction}\label{s1}
It is well-known that the  Cahn-Hilliard (CH) equation
%$$
\begin{equation}\label{0.CH}
\Dt u+\Dx(\Dx u-f(u)+g)=0,\ \ x\in\Omega, \ \ u\big|_{t=0}=u_0,
\end{equation}
%$$
where the unknown function $u=u(t,x)$ is the so-called order parameter, $f$ is a given non-linear interaction function and $g$ are the given external forces, is central for the material sciences and many papers are devoted to the mathematical study of this and related equations, see \cite{BaE,CH,CMZ,D,EK,EK1,EMZ1,Ell,ElS,MZ3,MZ2,MZ1,No1} and references therein. The most studied is the case where the underlying domain $\Omega$ is bounded. In this case a more or less complete theory of this equation including the global well-posedness, dissipativity, existence of finite dimensional global and exponential attractors for various classes of non-linearities (e.g., for fast growing or even singular ones) are available in the literature, see \cite{CMZ,MZ,No1} and references therein.
\par
The case where the underlying domain is unbounded is more difficult and less understood. The main problem here is the loss of dissipation in the low frequency limit $u_0=u_0(\eb x)$, $\eb\to0$, which prevents the solution semigroup to be dissipative in a classical sense even in the case of finite energy solutions, see \cite{DKS} for some partial results on the long-time behavior of solutions for the viscous CH equations in the whole space and \cite{CM,PZ} for the non-dissipative bounds in the case of infinite-energy solutions. We mention also the case of  CH equation in a cylindrical domain   with Dirichlet boundary conditions considered in \cite{Bo,EKZ} where the dissipation is guaranteed by the Poincare inequality.
\par
Therefore, in order to restore the dissipation mechanism, it seems reasonable to consider the physically relevant modifications of the initial CH model. One of the most convenient  from the mathematical point of view model for that is the  Cahn-Hilliard-Oono  (CHO) equation
%$$
\begin{equation}\label{0.CHO}
\Dt u+\Dx(\Dx u-f(u)+g)+\alpha u=0,\ \ u\big|_{t=0}=u_0,
\end{equation}
%$$
where the extra term $\alpha u$ with $\alpha>0$ models the non-local long-ranged interactions, see \cite{Miro,Oop} for more details. The extra dissipation term $\alpha u$ does not change the type of the equation and does not affect much the analytical properties of the solutions on the finite time interval. However, the presence of this extra dissipation removes the problem with low frequency modes and allows us to apply the weighted energy theory and verify the global well-posedness and dissipativity even in the case of infinite energy solutions, see \cite{PZ} for more details.
\par
Another interesting modification is the so-called hyperbolic relaxation of the CH equation:
%$$
\begin{equation}\label{0.hCH}
\eb\Dt^2 u+\Dt u+\Dx(\Dx u-f(u)+g)=0,\ \ \eb>0,\  \ u\big|_{t=0}=u_0,\ \ \Dt u\big|_{t=0}=u_1
\end{equation}
%$$
which has been introduced by P. Galenko in order to treat in a more accurate way the non-equilibrium effects in spinodal decomposition, see \cite{GJ,GL1,GL2,GL3}. In contrast to the previous modification, the inertial term $\eb\Dt^2 u$ changes drastically the type of the equation (from parabolic to hyperbolic) and the analytical properties of its solutions. Moreover, the nonlinearity $\Dx(f(u))$ becomes "critical" even if the equation is considered in the class of smooth solutions and "supercritical" if the estimate $u(t)\in L^\infty(\Omega)$ is not available. By this reason, despite the big current interest to this equation, see e.g., \cite{GGMP1D,GGMP3D,GGMP1DM,GSZ,GSZ1,GPS} and references therein, a considerable mathematical theory (which includes global well-posedness, dissipativity, asymptotic smoothness, etc.) exists only in space dimension one or two, see \cite{GGMP1D,GGMP1DM,GSZ}. Moreover, to the best of our knowledge, the global well-posedness of solutions for this problem in 3D is not known even in the case of bounded domain $\Omega$ and globally bounded non-linearity $f$, so the so-called weak trajectory attractors are used to study the long-time behavior of the considered equation, see \cite{CV,GSZ3}. Some exception is the case of very small $\eb>0$ where the smooth dissipative solutions can be constructed using the perturbation arguments, see \cite{GSZ1,GPS}.
\par
The main aim of the present paper is to combine the CHO equation and the hyperbolic relaxation mechanism mentioned above and consider the {\it hyperbolic} CHO equation
%$$
\begin{equation}\label{0.hCHO}
\Dt^2u+\Dt u+\Dx(\Dx u-f(u)+g)+\alpha u=0
\end{equation}
%$$
in the whole space $\Omega=\R^3$. In contrast to the previous contributors, we use not only the proper energy estimates, but also the so-called Strichartz estimates which are known to be crucial to study the non-linear wave or Scr\"odinger equations, see \cite{stri,plan2,plan, sogge-book, straus,tao} and references therein. In particular, using the known Stricharts estimates for the Schr\"odinger equations, we deduce the crucial control $u\in L^4(t,t+1;L^\infty(\R^3))$
for the solutions of \eqref{0.hCHO} which is sufficient to verify the global well-posedness in the case of sub-quintic nonlinearity $f$:
%$$
\begin{equation}\label{0.f}
|f''(u)|\le C(1+|u|^{3-\kappa}),\ \ \kappa>0,
\end{equation}
%$$
satisfying some natural dissipativity assumptions. To be more precise, we consider the hyperbolic CHO equation in the energy phase space
%$$
\begin{equation}
\Cal E:=[\dot H^1(\R^3)\cap\dot H^{-1}(\R^3)]\times \dot H^{-1}(\R^3),\ \ \xi_u:=(u,\Dt u)\in\Cal E
\end{equation}
%$$
and restrict ourselves to consider only the energy solutions $\xi_u\in C(0,T;\Cal E)$ which possesses the additional regularity $u\in L^4(0,T;L^\infty(\R^3))$ which we referred as the Strichartz solutions. Here and below $\dot H^s(\R^3)$ stands for the homogeneous Sobolev space of order $s$ in $\R^3$.  The main result of the present paper is the following theorem.

\begin{theorem}\label{Th0.main} Let the non-linearity $f$ satisfy the sub-quintic growth restriction and some dissipativity assumptions stated in \eqref{2.f}. Assume also that the extrenal force $g\in\dot H^1(\R^3)$. Then, for any initial data $\xi_0=(u_0,u_1)\in\Cal E$, equation \eqref{0.hCHO} possesses a unique global Strichartz solution $u(t)$ and this solution satisfies the following dissipative estimate:
%$$
\begin{equation}\label{0.dis}
\|\xi_u(t)\|_{\Cal E}+\|u\|_{L^4(t,t+1;L^\infty)}\le Q(\|\xi_u(0)\|_{\Cal E})e^{-\beta t}+Q(\|g\|_{\dot H^1}),
\end{equation}
%$$
where the positive constant $\beta$ and monotone function $Q$ are independent of the solution $u$ and  time $t\ge0$.
\par
Moreover, the Strichartz solution semigroup $S(t):\Cal E\to\Cal E$ generated by equation \eqref{0.hCHO} possesses a global attractor in the energy space $\Cal E$ and this attractor is bounded in the more regular space
%$$
\begin{equation}
\Cal E_2:=[\dot H^3(\R^3)\cap\dot H^{-1}(\R^3)]\times[\dot H^1(\R^3)\cap\dot H^{-1}(\R^3)].
\end{equation}
%$$
\end{theorem}
Note that in the case of damped wave equations the similar results have been recently extended to the case of general bounded domains and even with the critical quintic nonlinearity, see \cite{stri,plan2,KSZ2013}. However, in contrast to the wave equation (with the solution operator $e^{it(-\Dx)^{1/2}}$), the Strichartz estimates are more complicated for the case of the Scr\"odinger equation (with the solution operator $e^{it\Dx}$), and cannot be extended to general bounded domains without loss of regularity, see \cite{sogge1} and references therein. In particular, this loss of regularity does not allow to obtain the desired control of the $L^4(L^\infty)$ norm of the solution and prevents the straightforward extension of Theorem \ref{Th0.main} to the case of bounded domains, see Remark \ref{Rem5.stri} for more details.
\par
The paper is organized as follows.
\par
The classical Strichartz estimate for the linear Scr\"odinger equations is reminded in Section \ref{s2}. Based on this estimate, we  deduce its analogues for the case of plate equation as well as for the linear hyperbolic CHO equation (with $f=0$). The estimates obtained in this section are the key technical tools for our study of the non-linear hyperbolic CHO equation.
\par
The global well-posedness for the Strichartz solutions of the non-linear hyperbolic CHO equation is verified in Section \ref{s3}. Moreover, the dissipative estimate \eqref{0.dis} is verified there. Note that the fact that the non-linearity $f$ has a  sub-quintic growth rate is essentially used here. Similarly to \cite{KSZ2013}, we verify the control of the Strichartz norm through the energy norm of the solution and gain then the dissipativity of the Strichartz norm just from the straightforward energy estimate. Note also that the well-posedness is proved only in the class of Strichartz solution and we do not know whether or not any energy solution is the Strichartz one.
\par
The asymptotic compactness of the Strichartz solution semigroup $S(t)$ in the energy space $\Cal E$  is verified in Section \ref{s4} using the so-called energy method, see \cite{Ball,MRW} for more details.
\par
The further regularity of the global attractor $\Cal A$ (its boundedness in $\Cal E_2$) is verified in Section \ref{s5}. Note that, even with the Strichartz estimate, the hyperbolic CHO equation remains "critical" and the proof of the further regularity of the attractor is far from being trivial and the standard methods seem not working here. We overcome this difficulty by using in a crucial way the asymptotic compactness verified in the previous section.
\par
Finally, in the concluding Section \ref{s6}, we discuss various generalizations and open problems related with the proved results.

\section{Preliminaries: Energy and Strichartz estimates for the linear plate equations}\label{s2}
The aim of this section is to introduce a number of estimates for the linear Cahn-Hilliard-Oono (CHO) equation which are crucial for what follows. We start with recalling the known estimates for the linear homogeneous Schr\"odinger equation which play the central role in deriving the Strichartz type estimates for the CHO equation. Namely, consider the following equation in the whole space $\R^3$:
%$$
\begin{equation}\label{1.schro}
\Dt U-i\Dx U=H(t),\ \ U\big|_{t=0}=U_0,
\end{equation}
%$$
where $U=U(t,x)$ is the unknown complex valued function and $H=H(t,x)$ and $U_0=U_0(x)$ are given functions. We assume that
%$$
\begin{equation}\label{1.schro-ini}
U_0\in \dot H^1(\R^3),\ \ H\in L^1_{loc}(\R,\dot H^1(\R^3)).
\end{equation}
%$$
Here and below, we denote by $\dot H^s(\R^3)$ the  homogeneous Sobolev spaces of order $s$, in contrast to this, the usual non-homogeneous Sobolev spaces will be denoted by $H^s(\R^3)$ or $W^{s,p}(\R^3)$ (if $p\ne2$), see \cite{triebel} for more details concerning these spaces.
\par
The solution of equation \eqref{1.schro} can be written (at least formally) using the variation of constants formula:
%$$
\begin{equation}\label{1.schro-var}
U(t)=e^{it\Dx}U_0+\int_0^t e^{i\Dx(t-s)}H(s)\,ds.
\end{equation}
%$$
It is also well-known that formula \eqref{1.schro-var} defines indeed a unique solution $U\in C_{loc}(\R,\dot H^1(\R^3))$ of equation \eqref{1.schro}. More delicate is the following Strichartz estimate for $U$.
\begin{proposition}\label{Prop1.Str-schro} Let the initial data satisfies \eqref{1.schro-ini}. Then, the solution $U$ of problem \eqref{1.schro}
belongs to $L^4_{loc}(\R,L^{\infty}(\R^3))$ and, for any $T>0$, the following estimate holds:
%$$
\begin{equation}\label{1.schro-stri}
\|U\|_{C(-T,T;\dot H^1)}+\|U\|_{L^4(-T,T;L^\infty)}\le C_T(\|U_0\|_{\dot H^1}+\|H\|_{L^1(-T,T;\dot H^1)}),
\end{equation}
where the constant $C_T$ may depend on $T$, but is independent of $U_0$ and $H$.
\end{proposition}
For the proof of this result, see \cite{sogge1}.
\par
The next corollary which establishes that $U\in L^4_{loc}(\R,C(\R))$ will be essentially used in the sequel for verifying the regularity of the global attractor for the hyperbolic CHO equation.
\begin{corollary}\label{Cor1.schro-cont} Let the assumptions of Proposition \ref{Prop1.Str-schro} hold. Then the solution of equation \eqref{1.schro} satisfies
%$$
\begin{equation}\label{1.schro-cont}
U\in L^4(-T,T;C(\R^3))
\end{equation}
%$$
for any $T>0$.
\end{corollary}
\begin{proof} Indeed, let $\chi_N(\xi)$ be the characteristic function of a disk $\{N^{-1}\le|\xi|\le N\}$ and let the projector $P_N:\dot H^s\to\dot H^s$ be defined as follows:
%$$
\begin{equation}\label{1.pro}
(P_N f)(x):=\Cal F^{-1}_{\xi\to x}(\chi_N\Cal F_{x\to\xi}f),
\end{equation}
%$$
where $\Cal F_{x\to\xi}$ is a Fourier transform in $\R^3$. Let also $U_N:=P_N U$. This function obviously solves
%$$
\begin{equation}\label{1.smooth}
\Dt U_N-i\Dx U_N=H_N(t):=P_N H(t),\ \ U_N\big|_{t=0}=P_N U_0.
\end{equation}
%$$
Then, on the one hand, $U_N$ is $C^\infty$-smooth in $x$ (since $P_NU_0$ and $P_NH$ are smooth in $x$), so $U_N\in L^4(-T,T;C(\R^3))$ for all $N$. On the other hand, since $P_NU_0\to U_0$ in $\dot H^1$ and $P_NH\to H$ in $L^1(-T,T;\dot H^1)$, applying estimate \eqref{1.schro-stri} to the function $U_N-U_M$, we see that $U_N$ is a Cauchy sequence in $L^4(-T,T;C(\R^3))$. Thus, \eqref{1.schro-cont} is proved and the lemma is also proved.
\end{proof}
\begin{remark}\label{Rem1-C0} Arguing analogously, we see that the space $C(\R^3)$ in \eqref{1.schro-cont} can be replaced by the space $C_0(\R^3)$ consisting of all continuous functions on $\R^3$ tending to zero as $|x|\to\infty$.
\end{remark}
We now switch to the so-called plate equation of the form
%$$
\begin{equation}\label{1.plate}
\Dt^2 V+\Dx^2 V=\Dx H(t),\  V\big|_{t=0}=V_0,\ \ \Dt V\big|_{t=0}=V_1,
\end{equation}
%$$
where
%$$
\begin{equation}\label{1.plate-ini}
V_0\in\dot H^1(\R^3),\ \ V_1\in\dot H^{-1}(\R^3),\ \ H\in L^1_{loc}(\R,\dot H^1(\R^3)).
\end{equation}
%$$
Then, the formal multiplication of the equation by $\Dx^{-1}V_t$ gives the energy identity
%$$
\begin{equation}\label{1.plate-en}
\frac12\frac d{dt}(\|\Dt V\|^2_{\dot H^{-1}}+\|V\|^2_{\dot H^1})=-(H,\Dt V)
\end{equation}
%$$
(here and below, we denote by $(U,V)$ the usual inner product in $L^2(\R^3)$). This identity gives in a standard way the well-posedness of the equation in the class of energy solutions
%$$
\begin{equation}\label{1.plate-sol}
\xi_V:=(V,\Dt V)\in C_{loc}(\R,\dot H^1\times \dot H^{-1})
\end{equation}
%$$
as well as the energy estimate
%$$
\begin{equation}\label{1.en-plate}
\|\xi_V(T)\|_{\dot H^1\times\dot H^{-1}}\le C_T(\|\xi_V(0)\|_{\dot H^1\times\dot H^{-1}}+\|H\|_{L^1(-T,T;\dot H^1)}),
\end{equation}
%$$
where $C_T$ may depend on $T$, but is independent of $V$, see the case of linear CHO equation below for more details. In order to obtain more delicate Strichartz type estimates for this equation, we write down the variation of constants formula:
%$$
\begin{equation}\label{1.plate-var}
V(t)=\sin(\Dx t)\Dx^{-1}V_1+\cos(\Dx t)V_0+\int_0^t\sin(\Dx(t-s))H(s)\,ds.
\end{equation}
%$$
Using this formula and the result of Proposition \ref{Prop1.Str-schro} for the Schr\"odinger equation, we obtain the following result.
\begin{corollary}\label{Cor1.Str-plate} Let the assumptions \eqref{1.plate-ini} be satisfied. Then the solution $V$ of the plate equation \eqref{1.plate} belongs to the space $L^4_{loc}(\R,C(\R^3))$ and the following estimate holds for any $T>0$:
%$$
\begin{multline}\label{1.plate-stri}
\|\xi_V(T)\|_{\dot H^1\times\dot H^{-1}}+\|V\|_{L^4(-T,T;L^\infty)}\le\\\le C_T(\|\xi_V(0)\|_{\dot H^1\times\dot H^{-1}}+\|H\|_{L^1(-T,T;\dot H^1)}),
\end{multline}
%$$
where $C_T$ may depend on $T$, but is independent of $V$.
\end{corollary}
\begin{proof}
Indeed, using that $\sin(z)=\frac1{2i}(e^{iz}-e^{-iz})$ and $\cos(z)=\frac12(e^{iz}+e^{-iz})$ together with formula \eqref{1.schro-var} and Proposition \ref{Prop1.Str-schro}, we see that every term in \eqref{1.plate-var} belongs to $L^4(-T,T;C(\R^3))$, so the Strichartz part of estimate \eqref{1.plate-stri} holds. The energy part of it follows from \eqref{1.en-plate} and the corollary is proved.
\end{proof}
We now turn to the linear hyperbolic CHO equation of the form
%$$
\begin{equation}\label{1.lin-CHO}
\Dt^2 u+\Dt u+\Dx(\Dx u-H(t))+\alpha u=0,\ \ \xi_u\big|_{t=0}=\xi_0,
\end{equation}
%$$
where $\alpha>0$ is a given parameter. The energy equality for this equation formally reads
%$$
\begin{equation}\label{1.LCHO-en}
\frac12\frac d{dt}(\|\Dt u\|^2_{\dot H^{-1}}+\|u\|^2_{\dot H^1}+\alpha \|u\|^2_{\dot H^{-1}})+\|\Dt u\|^2_{\dot H^{-1}}=-(H(t),\Dt u).
\end{equation}
%$$
This guesses that the phase space for the energy solutions should be
%$$
\begin{equation}\label{1.phase}
\Cal E:=(\dot H^1\cap \dot H^{-1})\times\dot H^{-1},\ \|\xi_u\|_{\Cal E}^2:=\|\Dt u\|^2_{\dot H^{-1}}+\|u\|^2_{\dot H^1}+\alpha \|u\|^2_{\dot H^{-1}}
\end{equation}
%$$
and the energy solution of \eqref{1.lin-CHO} is a function $\xi_u\in C(0,T;\Cal E)$ which satisfies the equation as an equality in $\dot H^{-1}+\dot H^{-3}$. The next proposition gives the existence of such solution as well as the validity of the energy identity \eqref{1.LCHO-en}.

\begin{proposition}\label{Prop1.en} Let $H\in L^1(0,T;\dot H^1)$ and the initial data $\xi_0\in \Cal E$. Then, there exists a unique energy solution $u$ of equation \eqref{1.lin-CHO} and the energy identity \eqref{1.LCHO-en} holds for almost all $t\ge0$.
\end{proposition}
\begin{proof} Note that $H\in L^1(0,T;\dot H^1)$ and $\Dt u\in L^\infty(0,T;\dot H^{-1})$, so the term $(H,\Dt u)$ is well-defined for any energy solution. However, in order to obtain \eqref{1.LCHO-en}, we need to multiply the equation by $\Dx^{-1}\Dt u$ and the terms $(\Dt^2u,\Dx^{-1}\Dt u)$ and $(\Dx u,\Dt u)$ are a priori not well-defined, so we need to use the standard approximation arguments for justifying it. Namely, let the projector $P_N$ be the same as in the proof of Corollary \ref{Cor1.schro-cont} and $u_N:=P_N u$. Then, this function is smooth in $x$ and solves equation \eqref{1.lin-CHO} with $H$ replaced by $H_N:=P_NH$. Thus, writing the energy identity in the equivalent integral form for $u_N$, we have
%$$
\begin{equation}\label{1.en-app}
\frac12(\|\xi_{u_N}(t)\|^2_{\Cal E}-\|\xi_{u_N}(s)\|^2_{\Cal E})=-\int_{s}^t(H_N(\tau),\Dt u_N(\tau))+\|\Dt u_N(\tau)\|^2_{\dot H^{-1}}\,d\tau,
\end{equation}
%$$
for all $0\le s\le t$. Passing to the limit $N\to\infty$ in this relation, we end up with the integral equality equivalent to \eqref{1.LCHO-en}. Thus, the energy identity is verified and the existence of a solution can be then proved using e.g. the Galerkin approximation scheme based on the projectors $P_N$, so the proposition is also proved.
\end{proof}
The next corollary gives the {\it dissipative} estimate for the solutions of \eqref{1.lin-CHO}.

\begin{corollary}\label{Cor1.lin-endis} Let the assumptions of Proposition \ref{Prop1.en} hold. Then, the solution $u$ of equation \eqref{1.lin-CHO}
satisfies the following estimate:
%$$
\begin{equation}\label{1.lin-dis}
\|\xi_u(t)\|_{\Cal E}\le C\|\xi_u(0)\|_{\Cal E}e^{-\beta t}+C\int_0^te^{-\beta(t-s)}\|H(s)\|_{\dot H^1}\,ds,
\end{equation}
%$$
where the positive constants $C$ and $\beta$ are independent of $u$ and $t$.
\end{corollary}
\begin{proof} Indeed, multiplying equation \eqref{1.lin-CHO} by $\delta\Dx^{-1} u$, where $\delta>0$ is sufficiently small (which is allowed since $u\in \dot H^1\cap \dot H^{-1}$), we have
%$$
\begin{equation}
\delta\frac d{dt}((-\Dx^{-1}\Dt u,u)+\frac12\|u\|^2_{\dot H^{-1}})-\delta\|\Dt u\|^2_{\dot H^{-1}}+\delta\|u\|^2_{\dot H^{1}}+\alpha\delta\|u\|^2_{\dot H^{-1}}=-\delta(H,u).
\end{equation}
%$$
Taking a sum of this relation with the energy identity \eqref{1.LCHO-en}, we get
%$$
\begin{multline}\label{1.lin-big}
\frac d{dt}(\frac12\|\xi_u\|^2_{\Cal E}+\delta(\Dt u,-\Dx^{-1}u)+\frac\delta2\|u\|^2_{\dot H^{-1}})+\\+(1-\delta)\|\Dt u\|^2_{\dot H^{-1}}+\delta\|u\|^2_{\dot H^1}+\alpha\delta\|u\|^2_{\dot H^{-1}}=-(H,\Dt u+\delta u).
\end{multline}
%$$
Introducing $\Cal E_u(t):=\frac12\|\xi_u\|^2_{\Cal E}+\delta(\Dt u,-\Dx^{-1}u)+\frac\delta2\|u\|^2_{\dot H^{-1}}$, we see that for sufficiently small $\delta>0$,
%$$
\begin{equation}
C^{-1}\|\xi_u(t)\|^2_{\Cal E}\le \Cal E_u(t)\le C\|\xi_u(t)\|^2_{\Cal E}
\end{equation}
%$$
for some $C>0$. On the other hand, it follows from \eqref{1.lin-big} that, for sufficiently small $\beta$,
%$$
\begin{equation}
\frac d{dt}\Cal E_u(t)+\beta\Cal E_u(t)\le C\|H(t)\|_{\dot H^1}\Cal E_u(t)^{1/2}
\end{equation}
%$$
and the Gronwall inequality gives the desired estimate \eqref{1.lin-dis} and finishes the proof of the corollary.
\end{proof}
The next corollary combines the obtained energy estimate with the Strichartz estimate for the plate equation.

\begin{corollary}\label{Cor1.lin-Str} Let the assumptions of Proposition \ref{Prop1.en} hold. Then, the solution $u$ of equation \eqref{1.lin-CHO}
belongs to $L^4(0,T;C(\R^3))$ for all $T>0$ and satisfies the following estimate:
%$$
\begin{multline}\label{1.lin-str}
\|\xi_u(t)\|_{\Cal E}+\|u\|_{L^4(\max\{0,t-1\},t;L^\infty)}\le\\\le C\|\xi_u(0)\|_{\Cal E}e^{-\beta t}+C\int_0^te^{-\beta(t-s)}\|H(s)\|_{\dot H^1}\,ds,
\end{multline}
%$$
where the positive constants $C$ and $\beta$ are independent of $u$ and $t$.
\end{corollary}
\begin{proof} We interpret the hyperbolic CHO equation as a linear plate equation
%$$
\begin{equation}\label{1.l-plate}
\Dt^2 u+\Dx^2 u=\Dx \tilde H(t),\ \ \tilde H(t):=H(t)-\Dx^{-1}(\Dt u(t)+\alpha u(t)).
\end{equation}
%$$
Then, due to the energy estimate,
%$$
\begin{equation}
\|\tilde H\|_{L^1(\max\{0,t-1\},t;\dot H^1)}\le C(\|H\|_{L^1(\max\{0,t-1\},t;\dot H^1)}+\|\xi_u(\max\{0,t-1\})\|_{\Cal E})
\end{equation}
%$$
and, due to the Strichartz estimate \eqref{1.plate-stri} for the plate equation \eqref{1.l-plate}, we have
%$$
\begin{equation}
\|u\|_{L^4(\max\{0,t-1\},t;L^\infty)}\le C(\|H\|_{L^1(\max\{0,t-1\},t;\dot H^1)}+\|\xi_u(\max\{0,t-1\})\|_{\Cal E})
\end{equation}
%$$
and the desired estimate \eqref{1.lin-str} is an immediate corollary of this estimate and the dissipative energy estimate \eqref{1.lin-dis}. Thus, the corollary is proved.
\end{proof}

\section{The non-linear equation: global well-posedness and dissipativity}\label{s3}
This section is devoted to study the nonlinear hyperbolic CHO equation
%$$
\begin{equation}\label{2.CHO}
\Dt^2 u+\Dt u+\Dx(\Dx u-f(u)+g)+\alpha u=0,\ \ \xi_u\big|_{t=0}=\xi_0.
\end{equation}
%$$
Remind that the initial data is assumed to belong to the energy space $\Cal E=(\dot H^1\cap\dot H^{-1})\times \dot H^1$, the external force
$g$ lives in the space $\dot H^1$ and the nonlinearity $f\in C^2(\R,\R)$ satisfies the following dissipativity and growth restrictions
%$$
\begin{equation}\label{2.f}
\begin{cases}
1.\ \ f(u).u\ge0;\\
2.\ \ F(u)\le Lf(u).u+K|u|^2;\\
3.\ \ |f''(u)|\le C(1+|u|^{3-\kappa});
\end{cases}
\end{equation}
%$$
where all constants involved are positive and $0<\kappa\le 3$ and $F(u):=\int_0^uf(v)\, dv$.
\par
We start with defining the proper class of solutions for equation \eqref{2.CHO}.
\begin{definition}\label{Def2.sol} A function $u=u(t,x)$ is a {\it Strichartz} solution of problem \eqref{2.CHO} if, for any $T>0$,
%$$
\begin{equation}
\xi_u\in C(0,T;\Cal E),\ \ u\in L^4(0,T;C(\R^3))
\end{equation}
%$$
and equation \eqref{2.CHO} is satisfied as an equality in $\dot H^{-1}+\dot H^{-3}$.
\end{definition}
Note that, due to the growth restriction on $f$, we have
%$$
\begin{equation}
\|f(u)\|_{H^1}\le C(1+\|u\|^4_{L^\infty})(1+\|u\|_{H^1}),
\end{equation}
%$$
so, using the obvious embeddings
%$$
\begin{equation}\label{2.emb}
H^1(\R^3)\subset \dot H^1(\R^3),\ \ \dot H^{-1}(\R^3)\cap \dot H^1(\R^3)\subset H^1(\R^3),
\end{equation}
%$$
we see that $f(u)\in L^1(0,T; \dot H^1(\R^3))$ for any Strichartz solution $u$ of problem \eqref{2.CHO}. Thus, we may interpret the terms $f(u)+g$ as the external force for the linear CHO equation \eqref{1.lin-CHO}. Then, due to Proposition \ref{Prop1.en}, we have the energy identity
%$$
\begin{equation}
\frac12\frac d{dt}\|\xi_u\|^2_{\Cal E}+\|\Dt u\|^2_{\dot H^{-1}}=-(f(u),\Dt u)+(g,\Dt u)=-\frac d{dt}((F(u),1)-(g,u)).
\end{equation}
%$$
Therefore, any Strichartz solution of problem \eqref{2.CHO} satisfied the energy identity
%$$
\begin{equation}\label{2.energy}
\frac12\frac d{dt}\(\|\xi_u\|^2_{\Cal E}+(F(u),1)-(g,u)\)+\|\Dt u\|^2_{\dot H^{-1}}=0.
\end{equation}
%$$
Note also that, due to the first assumption of \eqref{2.f}, $f(0)=0$, so, taking into the account the third assumption, we see that
%$$
\begin{equation}\label{2.F}
|F(u)|+f(u).u\le C(|u|^2+|u|^6)
\end{equation}
%$$
and the terms $(F(u),1)$ and $(f(u),u)$ are well-defined for any energy solution $u$. The next proposition gives the analogue of dissipative energy estimate for the nonlinear hyperbolic CHO equation.
\begin{proposition}\label{Prop2.en-dis} Let the assumptions \eqref{2.f} hold, the external force $g\in\dot H^1(\R^3)$ and $u$ be a Strichartz solution of the hyperbolic CHO equation \eqref{2.CHO}. Then, the following dissipative energy estimate holds:
%$$
\begin{equation}\label{2.en-dis}
\|\xi_u(t)\|_{\Cal E}\le Q(\|\xi_u(0)\|_{\Cal E})e^{-\beta t}+Q(\|g\|_{\dot H^1}), \ \ t\ge0,
\end{equation}
%$$
where the positive constant $\beta$ and monotone function $Q$ are independent of $t\ge0$ and the solution $u$.
\end{proposition}
\begin{proof} As in the proof of Corollary \ref{Cor1.lin-endis}, we multiply equation \eqref{2.CHO} by $\Dx^{-1}(\Dt u+\delta u)$ (the multiplication on $\Dx^{-1}u$ is allowed and the multiplication on $\Dx^{-1}\Dt u$ is justified above) and get
%$$
\begin{multline}\label{2.big}
\frac d{dt}(\frac12\|\xi_u\|^2_{\Cal E}+\delta(\Dt u,-\Dx^{-1}u)+\frac\delta2\|u\|^2_{\dot H^{-1}}+(F(u),1))+\\+(1-\delta)\|\Dt u\|^2_{\dot H^{-1}}+\delta\|u\|^2_{\dot H^1}+\alpha\delta\|u\|^2_{\dot H^{-1}}+\delta(f(u),u)=-(g,\Dt u+\delta u).
\end{multline}
%$$
Using the first assumption of \eqref{2.f}, we see that $F(u)\ge0$ and together with \eqref{2.F}, we have
%$$
\begin{equation}\label{2.eneq}
C^{-1}\|\xi_u(t)\|_{\Cal E}^2\le \bar {\Cal E}_u(t)\le Q(\|\xi_u(t)\|_{\Cal E}),
\end{equation}
%$$
where $\bar{\Cal E}_u(t):= \frac12\|\xi_u(t)\|^2_{\Cal E}+\delta(\Dt u(t),-\Dx^{-1}u(t))+\frac\delta2\|u(t)\|^2_{\dot H^{-1}}+(F(u(t)),1)$ and $\delta>0$ is small enough. On the other hand, using the second assumption of \eqref{2.f}, we deduce from \eqref{2.eneq} that for sufficiently small $\beta>0$,
%$$
\begin{equation}\label{2.gron}
\frac d{dt}\bar{\Cal E}_u(t)+\beta\bar{\Cal E}_u(t)\le C\|g\|^2_{\dot H^1}
\end{equation}
%$$
and the Gronwall inequality applied to this relation gives the desired estimate \eqref{2.en-dis} and finishes the proof of the proposition.
\end{proof}
As the next step, we want to obtain the analogue of the dissipative estimate \eqref{2.en-dis} for the Strichartz norm of the solution $u$. To this end, we need the following key result.
\begin{proposition}\label{Prop2.str-en} Let the assumptions of Proposition \ref{Prop2.en-dis} hold and let $u$ be a Strichartz solution of problem \eqref{2.CHO}. Then, the following estimate is valid:
%$$
\begin{equation}\label{2.str-en}
\|u\|_{L^4(0,1;L^\infty)}\le Q(\|\xi_u(0)\|_{\Cal E})+Q(\|g\|_{\dot H^1}),
\end{equation}
where $Q$ is a monotone function which is independent of $u$.
\end{proposition}
\begin{proof} We treat the non-linearity $f(u)$ in \eqref{2.CHO} as an external force and apply \eqref{1.lin-str} on a small time interval $t\in[0,\tau]$ where $\tau\ll1$ will be fixed later. Then, we get
%$$
\begin{multline}
Y_u(\tau):=\|\xi_u\|_{C(0,\tau;\Cal E)}+\|u\|_{L^4(0,\tau;L^\infty)}\le\\\le C(\|\xi_u(0)\|_{\Cal E}+\|g\|_{\dot H^1}+\|f(u)\|_{L^1(0,\tau;\dot H^1)})
\end{multline}
%$$
for some constant $C$ which is independent of $u$ and $\tau$. Using the subcritical (subquintic)  growth restriction  for $f(u)$, see \eqref{2.f} assumption 3, together with the fact that $f(0)=0$, we estimate the norm of $f(u)$ as follows:
%$$
\begin{multline}\label{2.bigint}
\|f(u)\|_{L^1(0,t;\dot H^1)}\le C\|(1+|u|^{4-\kappa})|\Nx u|\|_{L^1(0,\tau;L^2)}\le\\\le C(\tau+\|u\|_{L^{4-\kappa}(0,\tau;L^\infty)}^{4-\kappa})\|u\|_{L^\infty(0,\tau;\dot H^1)}\le\\
\le C(\tau^\gamma+\tau^\gamma\|u\|_{L^4(0,\tau;L^\infty)}^4)\|\xi_u\|_{L^\infty(0,\tau;\Cal E)}\le C\tau^\gamma(1+Y_u(\tau)^4)Y_u(\tau),
\end{multline}
%$$
where $\gamma=\gamma(\kappa)\in(0,3/4]$ and $C$ are independent of $u$ and $\tau$. Thus, we end up with the following relation:
%$$
\begin{equation}\label{2.Y}
Y_u(\tau)\le C(\|\xi_u(0)\|_{\Cal E}+\|g\|_{\dot H^1})+ C\tau^\gamma(1+Y_u(\tau)^4)Y_u(\tau).
\end{equation}
%$$
Since $Y_u(0)=\|\xi_u(0)\|_{\Cal E}$ and $Y_u(\tau)$ is continuous in $\tau$, estimate \eqref{2.Y} guarantees (see, for example, \cite{sogge-book}, Chapter IV, Lemma 2.2) that there exists $\tau_0=\tau_0(\|\xi_u(0)\|_{\Cal E}+\|g\|_{\dot H^1})$ such that
 %$$
 \begin{equation}\label{2.YY}
 Y_u(\tau)\le 2C(\|\xi_u(0)\|_{\Cal E}+\|g\|_{\dot H^1}), \ \  \tau\le \tau_0.
 \end{equation}
 %$$
Important point here is that the constant $C$ is independent of $u$ and $\tau_0$ depends only on the {\it norms} of $\xi_u(0)$ and $g$. This  allows us to obtain the desired estimate \eqref{2.str-en} on a big time interval $t\in[0,1]$ just by iterating \eqref{2.YY} and using that the energy norm of $\xi_u(t)$ is under the control due to the energy estimate \eqref{2.en-dis}. Thus, the proposition is proved.
\end{proof}
\begin{corollary}\label{Cor2.stri-dis} Let the assumptions of Proposition \ref{Prop2.en-dis} hold and  $u$ be a Strichartz solution of problem \eqref{2.CHO}. Then, the following estimate is valid:
%$$
\begin{equation}\label{2.stri-dis}
\|\xi_u(t)\|_{\Cal E}+\|u\|_{L^4(t,t+1;L^\infty)}\le Q(\|\xi_u(0)\|_{\Cal E})e^{-\beta t}+Q(\|g\|_{\dot H^1}),
\end{equation}
%$$
where the positive constant $\beta$ and the monotone function $u$ are independent of $t$ and $u$.
\end{corollary}
Indeed, estimate \eqref{2.stri-dis} is an immediate corollary of the dissipative energy estimate \eqref{2.en-dis} and the control \eqref{2.str-en} of the Strichartz norm through the energy norm.
\par
We are now ready to state the main result of the section.
\begin{theorem}\label{Th2.main} Let the assumptions of Proposition \ref{Prop2.en-dis}. Then, for every $\xi_0\in\Cal E$, there exists a unique Strichartz solution $u$ of problem \eqref{2.CHO} and this solution satisfies the dissipative estimate \eqref{2.stri-dis}.
\end{theorem}
\begin{proof} Indeed, the dissipative estimate is already verified. Since the non-linearity $f$ is subcritical ($\kappa>0$), the local existence of a Srtrichartz solution is straightforward and can be done using the Banach contraction theorem. Moreover, the interval of existence depends only on the energy norm of the initial data. Since the energy norm is under the control due to \eqref{2.stri-dis}, the global existence follows by the extension of a local solution. Thus, we only need to verify the uniqueness.
\par
Let $u_1$ and $u_2$ be two Strichartz solutions of equation \eqref{2.CHO} and let $v=u_1-u_2$. Then, this function solves
%$$
\begin{equation}\label{2.dif}
\Dt^2 v+\Dt v+\Dx(\Dx v-[f(u_1)-f(u_2)])+\alpha v=0.
\end{equation}
%$$
Interpreting this equation as a linear hyperbolic CHO equation with the extrenal forces $f(u_1)-f(u_2)$ and applying estimate \eqref{1.lin-str} on a small time interval $t\in[0,\tau]$, we have
%$$
\begin{multline}\label{2.Yv}
Y_v(\tau):=\|v\|_{C(0,\tau;\Cal E)}+\|v\|_{L^4(0,\tau;L^\infty)}\le\\\le C\(\|\xi_v(0)\|_{\Cal E}+\|f(u_1)-f(u_2)\|_{L^1(0,\tau;\dot H^1)}\),
\end{multline}
%$$
where the constant $C$ is independent of $u_1$, $u_2$ and $\tau$. Using the growth restrictions \eqref{2.f} on the nonlinearity, we estimate the last term in the right-hand side of \eqref{2.Yv} as follows:
%$$
\begin{multline}\label{2.huge}
\|f(u_1)-f(u_2)\|_{L^1(0,\tau;\dot H^1)}\le C\|(1+|u_1|^{4-\kappa}+|u_2|^{4-\kappa})\Nx v\|_{L^1(0,\tau;L^2)}+\\+C\|(1+|u_1|^{3-\kappa}+|u_2|^{3-\kappa})|v|(|\Nx u_1|+|\Nx u_2|)\|_{L^1(0,\tau;L^2)}=I_1+I_2.
\end{multline}
%$$
Estimating the first term in the right-hand side  \eqref{2.huge} exactly as in \eqref{2.bigint}, we have
%$$
\begin{equation}\label{2.I1}
I_1\le C\tau^\gamma(1+\|u_1\|^4_{L^4(0,\tau;L^\infty)}+\|u_2\|^4_{L^4(0,\tau;L^\infty)})Y_v(\tau),
\end{equation}
%$$
where the constants $\gamma=\gamma(\kappa)>0$ and $C>0$ are independent of $\tau$, $u_1$ and $u_2$. The second term $I_2$ can be estimated as follows:
%$$
\begin{multline}\label{2.hhuge}
I_2\le C\(\tau^{3/4}+\|u_1\|_{L^{4-4\kappa/3}(0,\tau;L^\infty)}^{3-\kappa}+\|u_2\|_{L^{4-4\kappa/3}(0,\tau;L^\infty)}^{3-\kappa}\)\|v\|_{L^4(0,\tau;L^\infty)}\times\\
\(\|u_1\|_{L^\infty(0,\tau;\dot H^1)}+\|u_2\|_{L^\infty(0,\tau;\dot H^1)}\)\le C\tau^\gamma\(1+\|u_1\|^3_{L^4(0,\tau;L^\infty)}+\|u_2\|^3_{L^4(0,\tau;L^\infty)}\)\\\times\(\|u_1\|_{L^\infty(0,\tau;\dot H^1)}+\|u_2\|_{L^\infty(0,\tau;\dot H^1)}\)Y_v(\tau)
\end{multline}
%$$
for some positive constants $\gamma=\gamma(\kappa)$ ( the same as in \eqref{2.bigint}) and $C$ which are independent of $\tau$, $u_1$ and $u_2$. Inserting the obtained estimates into the right-hand side of \eqref{2.Yv} and using estimate \eqref{2.stri-dis} to control the norms of $u_1$ and $u_2$, we finally arrive at
%$$
\begin{equation}\label{2.YYv}
Y_v(\tau)\le C\tau^\gamma\(1+Q(\|\xi_{u_1}(0)\|_{\Cal E})+Q(\|\xi_{u_2}(0)\|_{\Cal E})\)Y_v(\tau)+C\|\xi_v(0)\|_{\Cal E}.
\end{equation}
%$$
This estimate shows that there exists $\tau_0>0$ depending only on the energy norms of the initial data for $u_1$ and $u_2$ such that
%$$
\begin{equation}\label{2.fin}
Y_v(\tau)\le C\|\xi_v(0)\|_{\Cal E},\ \ \tau\le\tau_0.
\end{equation}
%$$
Thus, the uniqueness is proved and the theorem is also proved.
\end{proof}
\begin{remark}\label{Rem2.fin} Actually, estimate \eqref{2.fin} gives a bit more than the uniqueness. Indeed, iterating this estimate and using that the energy norms of $u_1$ and $u_2$ are under the control, we have
%$$
\begin{equation}\label{2.lip}
\|\xi_{u_1}(t)-\xi_{u_2}(t)\|_{\Cal E}+\|u_1-u_2\|_{L^4(t,t+1;L^\infty)}\le Ce^{Kt}\|\xi_{u_1}(0)-\xi_{u_2}(0)\|_{\Cal E},
\end{equation}
%$$
where the constants $C$ and $K$ depend only on the energy norms for the intitial data for $u_1$ and $u_2$. Thus, we have for free the Lipschitz continuity of the Strichartz solution of \eqref{2.CHO} with respect to the initial data.
\par
Note also that, analogously to the case of bounded domains, see \cite{GSZ3}, we may define the class of so-called {\it energy} solutions which belong to the space $L^\infty(0,T;\Cal E)$ and even prove their global existence and dissipativity. However, for the uniqueness we have crucially used the extra regularity given by the Strichartz estimate. Moreover, to the best of our knowledge, the uniqueness of energy solutions is an open problem even in the case when the non-linearity $f$ is globally bounded. In particular, it is not known whether or not any energy solution is automatically a Strichartz one.
\end{remark}

\section{The global attractor}\label{s4}
As shown in the previous section, the hyperbolic CHO equation \eqref{2.CHO} is globally well-posed in the energy phase space $\Cal E$ in the class of Strichartz solutions. Thus, the solution semigroup
%$$
\begin{equation}
S(t):\Cal E\to\Cal E,\ \ S(t)\xi_0:=\xi_u(t),\ \ t\ge0,
\end{equation}
%$$
where $u(t)$ is a Strichartz solution of problem \eqref{2.CHO}, is well-defined. Moreover, due to estimate \eqref{2.stri-dis}, this semigroup is dissipative in $\Cal E$. The main aim of this section is to verify that this solution semigroup possesses a global attractor in $\Cal E$.
\par
We start with reminding some basic fact from the attractor's theory, see \cite{BV,Te} for more details.
\begin{definition} Let $S(t):\Cal E\to\Cal E$ be a semigroup.
A set $\Cal B$ is an {\it absorbing} set for this semigroup if, for any bounded set $B\subset\Cal E$ there exists $T=T(B)$ such that
%$$
\begin{equation}
S(t)B\subset\Cal B
\end{equation}
%$$
for all $t\ge T$.
\par
 A set $\Cal B$ is an {\it attracting} set for the semigroup $S(t)$ if, for any bounded set
$B\subset\Cal E$ and every neighbourhood $\Cal O(\Cal B)$ of the set $\Cal B$ there exists time $T=T(B,\Cal O)$ such that
%$$
\begin{equation}
S(t)B\subset\Cal O(\Cal B)
\end{equation}
%$$
for all $t\ge T$.
\end{definition}
\begin{definition}\label{Def3.attr} Let $S(t):\Cal E\to\Cal E$ be a semigroup. A set $\Cal A\subset\Cal E$ is a {\it global attractor} for the semigroup $S(t)$ if the following conditions are satisfied:
\par
1.\ The set $\Cal A$ is compact in $\Cal E$;
\par
2. The set $\Cal A$ is strictly invariant: $S(t)\Cal A=\Cal A$ for all $t\ge0$;
\par
3. The set $\Cal A$ is an attracting set for the semigroup $S(t)$.
\end{definition}
To state the criterion for the attractor's existence we need one more definition.
\begin{definition}\label{Def3.ass-comp} A semigroup $S(t):\Cal E\to\Cal E$ is {\it asymptotically compact} if for any bounded set $B\subset\Cal E$, any sequence of the initial data $\xi_n\in B$ and any sequence of times $t_n\ge0$ such that $t_n\to\infty$ as $n\to\infty$, the sequence
%$$
\begin{equation}
\{S(t_n)\xi_n\}_{n=1}^\infty
\end{equation}
%$$
is precompact in $\Cal E$.
\end{definition}
To verify the existence of a global attractor for the hyperbolic CHO equation, we will use the following version of the attractor's existence criterion, see \cite{BV,Te} for the proof.
\begin{proposition}\label{Prop3.attr} Let the semigroup $S(t):\Cal E\to\Cal E$ possess the following properties:
\par
1.\ The operators $S(t):\Cal E\to\Cal E$ are continuous in $\Cal E$ for every fixed $t$;
\par
2.\  The semigroup $S(t)$ possesses a bounded attracting set;
\par
3.\ The semigroup $S(t)$ is asymptotically compact.
\par
Then the semigroup $S(t)$ possesses a global attractor $\Cal A\subset \Cal E$ which is generated by all complete trajectories of the semigroup $S(t)$:
%$$
\begin{equation}
\Cal A=\Cal K\big|_{t=0},
\end{equation}
%$$
where $\Cal K\subset L^\infty(\R,\Cal E)$ consists of all bounded functions $u:\R\to\Cal E$ such that $S(h)u(t)=u(t+h)$ for all $t\in\R$ and $h\ge0$.
\end{proposition}
We are now ready to state the main result of this section.
\begin{theorem}\label{Th3.attr} Let the assumptions of Theorem \ref{Th2.main} hold. Then, the solution semigroup $S(t)$ associated with the hyperbolic CHO equation \eqref{2.CHO} possesses a global attractor $\Cal A$ in the energy phase space $\Cal E$ which is generated by all complete bounded Strichartz solutions of \eqref{2.CHO}
%$$
\begin{equation}\label{3.a-srtr}
\Cal A=\Cal K\big|_{t=0},
\end{equation}
%$$
where $\Cal K\subset C_b(\R,\Cal E)$ is a set of all Strichartz solutions of \eqref{2.CHO} which are defined for all $t\in\R$ and bounded.
\end{theorem}
\begin{proof} We need to check the assumptions of Proposition \ref{Prop3.attr}. The continuity of the operators $S(t)$ in $\Cal E$ for every fixed $t$ follows from estimate \eqref{2.lip}. The existence of a bounded attracting (and even absorbing) set for $S(t)$ is guaranteed by by the dissipative estimate \eqref{2.stri-dis}. Thus, we only need to check the asymptotic compactness.
\par
To verify the desired asymptotic compactness, we will use the so-called energy method, see \cite{Ball,MRW}. Indeed, let $\xi_n\in\Cal E$ be a bounded sequence, $t_n\to\infty$ be a sequence of times tending to infinity. We need to prove that $S(t_n)\xi_n$ is precompact in $\Cal E$. To this end, we define a sequence $u_n(t)$ of Strichartz solutions of the following problems:
%$$
\begin{equation}\label{3.CHOn}
\Dt^2 u_n+\Dt u_n+\Dx(\Dx u_n-f(u_n)+g)+\alpha u_n=0,\ \ \xi_{u_n}\big|_{t=-t_n}=\xi_n,\ \ t\ge-t_n.
\end{equation}
%$$
Thus, in order to prove the asymptotic compactness, we need to prove that the sequence $\{\xi_{u_n}(0)\}$ is precompact in $\Cal E$. We will do this in two steps. At step one we prove that up to a subsequence $\xi_{u_n}(0)$ converges {\it weakly} to $\xi_{u}(0)$ for some complete bounded solution $u\in\Cal K$ and at step 2 we show that the {\it energy} $\|\xi_{u_n}(0)\|_{\Cal E}$ converges to the energy $\|\xi_u(0)\|_{\Cal E}$ of the limit solution $u$ and this will give the desired strong convergence.
\par
{\it Step 1.} Since the sequence $\xi_n$ is bounded in $\Cal E$, the dissipative energy estimate \eqref{2.stri-dis} gives the uniform boundedness of the corresponding solutions $u_n$:
%$$
\begin{equation}\label{3.bounded}
\|\xi_{u_n}\|_{L^\infty(\R, \Cal E)}\le C,\ \ \|u_n\|_{L^4(T,T+1,L^\infty)}\le C,\ \ T\in\R,
\end{equation}
%$$
where $C$ is independent of $n$ (to simplify the notations, we assume that $u_n$ and $\Dt u_n$ are extended by zero for $t\le -t_n$). Thus, without loss of generality, we may assume that
%$$
\begin{equation}\label{3.weak}
\xi_{u_n}\rightharpoondown \xi_u \text{ weakly star in } \ L^\infty_{loc}(\R,\Cal E) \text{ and } u_n\rightharpoondown u \text{ weakly star in } L^4_{loc}(\R,L^\infty)
\end{equation}
%$$
and the limit function $u$ satisfies estimates \eqref{3.bounded} as well, see \cite{RR}. In order to verify that $u\in\Cal K$, we need to pass to the limit $n\to\infty$ in equations \eqref{3.CHOn}. As usual, the passage to the limit in the linear terms is immediate and only the non-linear term $f(u_n)$ may cause some problems. Thus, we only need to verify that, for every test function $\phi\in C_0^\infty(\R,\dot H^1\cap \dot H^{-1})$,
%$$
\begin{equation}\label{3.flim}
\int_\R(f(u_n(t)),\phi(t))\,dt\to \int_\R(f(u(t)),\phi(t))\,dt.
\end{equation}
%$$
Moreover, since $C_0^\infty(\R^3)$ is dense in $\dot H^1\cap\dot H^{-1}$, it is sufficient to verify \eqref{3.flim} for the test functions $\phi\in C_0^\infty(\R_t\times\R^3_x)$ only. From the uniform estimate \eqref{3.bounded}, we conclude that $f(u_n)$
 is bounded in $L^{1+\eb}(T,T+1;L^{1+\eb}(B^R_0))$ for some $\eb>0$ and every  $R>0$ and $T\in\R$ (here and below, we denote by $B^R_0$ the ball of radius $R$ centered at zero in $\R^3$). Thus, without loss of generality, we may assume that
 %$$
 \begin{equation}
 f(u_n)\rightharpoondown \xi \ \text{ weakly in }\ L^{1+\eb}_{loc}(\R_t\times\R^3_x)
 \end{equation}
 %$$
and, to verify the convergence \eqref{3.flim}, we only need to check that $\xi=f(u)$. In turn, to check the last identity, we only need to verify that
%$$
\begin{equation}\label{3.ae}
u_n(t,x)\to u(t,x)\ \ \text{ almost everywhere in $(t,x)\in\R^4$,}
\end{equation}
%$$
see e.g., \cite{BV,Te}. To verify \eqref{3.ae}, we note that, due to the obvious embeddings
%$$
\begin{equation}
   \dot H^1(\R^3)\cap\dot H^{-1}(\R^3)\subset H^1(\R^3),\ \ \dot H^{-1}(\R^3)\subset H^{-1}(\R^3)
\end{equation}
%$$
and the uniform boundedness \eqref{3.bounded}, we know that
%$$
\begin{equation}\label{3.bound1}
\|u_n\|_{L^\infty(\R, H^1(\R^3))}+\|\Dt u_n\|_{L^\infty(\R,H^{-1}(\R^3))}\le C
\end{equation}
%$$
uniformly with respect to $n\to\infty$. Let $\psi_R=\psi_R(x)\in C_0^\infty(\R^3)$ be the cut off function such that $\xi_R(x)=1$ for $x\in B^R_0$ and $\xi_R(x)=0$ for $x\notin B^{2R}_0$. Then \eqref{3.bound1} implies that
%$$
\begin{equation}
\|\xi_R u_n\|_{L^\infty(\R,H^1(B^{2R}_0))}+\|\Dt(\xi_R u_n)\|_{L^\infty(\R,H^{-1}(B^{2R}_0))}\le C
\end{equation}
%$$
for some $C$ which is independent of $n$. Thus, due to the compactness theorem, we may assume without loss of generality that
%$$
\begin{equation}\label{3.un-strong}
u_n\to u \ \ \text{ strongly in }\  C((T,T+1)\times L^2(B^R_0))
\end{equation}
%$$
for all $T\in\R$ and $R\in\R_+$. Thus, passing to the subsequence once more if necessary, we see that \eqref{3.ae} indeed holds. Consequently,
the convergence \eqref{3.flim} also takes place and we have proved that the limit function $u$ is a Strichartz solution of equation \eqref{2.CHO} and that $u\in\Cal K$.
\par
To complete Step 1, we need to verify that
%$$
\begin{equation}\label{3.weak1}
\xi_{u_n}(0)\rightharpoondown \xi_u(0) \ \ \text{ in } \Cal E.
\end{equation}
%$$
Actually, the weak convergence $u_n(0)\rightharpoondown u(0)$ is straightforward due to the proved strong convergence \eqref{3.un-strong} and the facts that $u_n(0)$ are uniformly bounded in $\dot H^1\cap \dot H^{-1}$ and that $C_0^\infty$ is dense in $\dot H^1\cap \dot H^{-1}$. Thus, we only need to check that
%$$
\begin{equation}
(\Dt u_n(0),\psi)\to (\Dt u(0),\psi) \ \ \text{ for every } \ \psi\in\dot H^1(\R^3).
\end{equation}
%$$
To verify this, it is enough to check the convergence for $\psi_N:=P_N\psi$ only (where $P_N$ is the projector introduced in the proof of Corollary \ref{Cor1.schro-cont}). To this end, we introduce the function $\Psi_n(t):=(\Dt u_n(t),\psi_N)$. Then, since $\psi_N\in\dot H^s$ for every $s\in\R$, we may test equation \eqref{3.CHOn} by $\psi_N$ and get
%$$
\begin{multline}
\frac d{dt}\Psi_n=-(\Dt u_n,\psi_N)-(u_n,\Dx^2\psi_N)+(f(u_n),\Dx\psi_N)-\\
(g,\Dx\psi_N)-\alpha(u_n,\psi_N).
\end{multline}
%$$
Therefore, due to the uniform bounds \eqref{3.bounded} for $u_n$ and the growth restriction \eqref{2.f} for the nonlinearity $f$,
%$$
\begin{equation}
\|\Psi_n\|_{L^\infty(\R)}+\|\frac d{dt}\Psi_n\|_{L^\infty(\R)}\le C,
\end{equation}
%$$
where the constant $C$ depends on $N$ and $\psi$, but is independent of $n$. Therefore, without loss of generality, we may assume that $\Psi_n\to\Psi:=(\Dt u,\psi_N)$ {\it strongly} in the space $C_{loc}(\R)$ and, in particular, that $\Psi_n(0)\to\Psi(0)$. Thus, the weak convergence \eqref{3.weak1} is verified and Step 1 is completed.
\par
{\it Step 2.} We now verify that $\|\xi_{u_n}(0)\|_{\Cal E}\to\|\xi_u(0)\|_{\Cal E}$. To this end, we  will use the following analogue of \eqref{2.big} for the solutions \eqref{3.CHOn}:
%$$
\begin{multline}\label{3.big}
\frac d{dt}\(\frac12\|\xi_{u_n}\|^2_{\Cal E}+\delta(\Dt u_n,-\Dx^{-1}u_n)+\frac\delta2\|u_n\|^2_{\dot H^{-1}}+(F(u_n),1)\)+\\
\beta\(\frac12\|\xi_{u_n}\|^2_{\Cal E}+\delta(\Dt u_n,-\Dx^{-1}u_n)+\frac\delta2\|u_n\|^2_{\dot H^{-1}}+(F(u_n),1)\)+\\+
(1-\delta-\frac\beta2)\|\Dt u_n\|^2_{\dot H^{-1}}+(\delta-\frac\beta2)\|u_n\|^2_{\dot H^1}+(\alpha\delta-\frac{\alpha\beta}{2}-\frac{\beta\delta}{2})\|u_n\|^2_{\dot H^{-1}}+\\+[\delta(f(u_n),u_n)-\beta (F(u_n),1)]+\delta\beta(\Dt u_n,\Dx^{-1}u_n)=-(g,\Dt u_n+\delta u_n).
\end{multline}
%$$
Introducing the functionals
%$$
\begin{equation}
\tilde {\Cal E}_{u_n}(t):=\frac14\|\xi_{u_n}(t)\|^2_{\Cal E}+\delta(\Dt u_n(t),-\Dx^{-1}u_n(t))+\frac\delta2\|u_n(t)\|^2_{\dot H^{-1}}+(F(u_n(t)),1)
\end{equation}
%$$
and
%$$
\begin{multline}
\Cal H_{u_n}(t):=(1-\delta-\frac\beta2)\|\Dt u_n(t)\|^2_{\dot H^{-1}}+(\delta-\frac\beta2)\|u_n(t)\|^2_{\dot H^1}+\\
(\alpha\delta-\frac{\alpha\beta}{2}-\frac{\beta\delta}{2})\|u_n(t)\|^2_{\dot H^{-1}}+\\+[\delta(f(u_n(t)),u_n(t))-\beta (F(u_n(t)),1)]+\delta\beta(\Dt u_n(t),\Dx^{-1}u_n(t))
\end{multline}
%$$
we rewrite the identity \eqref{3.big} in the following  form:
%$$
\begin{multline}\label{3.bigid}
\frac14\|\xi_{u_n}(0)\|^2_{\Cal E}+\tilde{\Cal E}_{u_n}(0)=\(\frac14\|\xi_{u_n}(-t_n)\|_{\Cal E}^2+\tilde{\Cal E}_{u_n}(-t_n)\)e^{-\beta t_n}-\\
-\int_{-t_n}^0e^{\beta s}\Cal H_{u_n}(s)\,ds-\int_{-t_n}^0e^{\beta s}(g,\Dt u_n(s)+\delta u_n(s))\,ds.
\end{multline}
%$$
We want to pass to the limit $n\to\infty$ in this identity. To this end, we first fix positive $\delta$ to be small enough that
%$$
\begin{equation}
\frac14\|\xi_{v}(t)\|^2_{\Cal E}+\delta(\Dt v,\Dx^{-1}v)+\frac\delta2\|v\|^2_{\dot H^{-1}}\ge0
\end{equation}
%$$
for all $v\in\Cal E$. Then, the weak convergence \eqref{3.weak1} implies that
%$$
\begin{multline}
\frac14\|\xi_{u}(0)\|^2_{\Cal E}+\delta(\Dt u(0),-\Dx^{-1}u(0))+\frac\delta2\|u(0)\|^2_{\dot H^{-1}}\le\\
\le\liminf_{n\to\infty}\( \frac14\|\xi_{u_n}(0)\|^2_{\Cal E}+\delta(\Dt u_n(0),-\Dx^{-1}u_n(0))+\frac\delta2\|u_n(0)\|^2_{\dot H^{-1}}\).
\end{multline}
%$$
Moreover, the convergence \eqref{3.un-strong} implies that $u_n(0,x)\to u(0,x)$ almost everywhere in $\R^3$. Then, the Fatou lemma together with the fact that $F(u)\ge0$ give
%$$
\begin{equation}
(F(u(0)),1)\le\liminf_{n\to\infty}(F(u_n(0)),1)
\end{equation}
%$$
and, therefore,
%$$
\begin{equation}\label{3.lime}
\tilde{\Cal E}_u(0)\le\liminf_{n\to\infty}\tilde{\Cal E}_{u_n}(0).
\end{equation}
%$$
The first term in the right-hand side of \eqref{3.bigid} tends to zero due to the facts that $t_n\to\infty$ and the energy norms of the initial data for $u_n$ are uniformly bounded. Moreover, due to the facts that the solutions $u_n$ are uniformly bounded in $L^\infty(\R,\Cal E)$ and converge weakly star in $L^\infty_{loc}(\R,\Cal E)$, we have the convergence in the third term:
%$$
\begin{equation}\label{3.limg}
\int_{-t_n}^0e^{\beta s}(g,\Dt u_n(s)+\delta u_n(s))\,ds\to\int_{-\infty}^0e^{\beta s}(g,\Dt u(s)+\beta u(s))\,ds.
\end{equation}
%$$
To pass to the limit in the term containing the functional $\Cal H$, we fix $\beta>0$ to be small enough that
%$$
\begin{equation}
\delta f(u)u-\beta F(u)+\beta Ku^2\ge0
\end{equation}
%$$
(which is possible to do due to the 3rd assumption of \eqref{2.f}) and that
%$$
\begin{multline}
(1-\delta-\frac\beta2)\|\Dt v\|^2_{\dot H^{-1}}+(\delta-\frac\beta2)\|v\|^2_{\dot H^1}+\\
+(\alpha\delta-\frac{\alpha\beta}{2}-\frac{\beta\delta}{2})\|v\|^2_{\dot H^{-1}}-\beta K\|v\|_{L^2}+\delta\beta(\Dt v(t),\Dx^{-1}v(t))\ge0
\end{multline}
%$$
for all $v\in\Cal E$. Then, using the weak star convergence of $u_n$ in $L^\infty_{loc}(\R,\Cal E)$ together with the convergence almost everywhere and the Fatou lemma, we conclude that
%$$
\begin{equation}\label{3.limh}
\int_{-\infty}^0e^{\beta s}\Cal H_u(s)\,ds\le\liminf_{n\to\infty}\int_{-t_n}^0e^{\beta s}\Cal H_{u_n}(s)\,ds.
\end{equation}
%$$
Thus, passing to the limit $n\to\infty$ in the identity \eqref{3.bigid} and using \eqref{3.lime},\eqref{3.limg} and \eqref{3.limh}, we have
%$$
\begin{multline}\label{3.bigid1}
\frac14\limsup_{n\to\infty}\|\xi_{u_n}(0)\|^2_{\Cal E}+\tilde{\Cal E}_{u}(0)\le \\\le -\int_{-\infty}^0e^{\beta s}\Cal H_{u}(s)\,ds-\int_{-\infty}^0e^{\beta s}(g,\Dt u(s)+\delta u(s))\,ds.
\end{multline}
%$$
On the other hand, writing the analogue of the identity \eqref{3.bigid} for the limit solution $u\in\Cal K$, we end up with
%$$
\begin{equation}\label{3.bigid2}
\frac14\|\xi_{u}(0)\|^2_{\Cal E}+\tilde{\Cal E}_{u}(0)=-\int_{-\infty}^0e^{\beta s}\Cal H_{u}(s)\,ds-\int_{-\infty}^0e^{\beta s}(g,\Dt u(s)+\delta u(s))\,ds.
\end{equation}
%$$
Thus,
%$$
\begin{equation}
\limsup_{n\to\infty}\|\xi_{u_n}(0)\|_{\Cal E}\le\|\xi_u(0)\|_{\Cal E}\le\liminf_{n\to\infty}\|\xi_{u_n}(0)\|_{\Cal E},
\end{equation}
%$$
where the inequality in the right-hand side follows from the weak convergence \eqref{3.weak1}. This is possible only if we have the convergence of the norms
%$$
\begin{equation}
\|\xi_{u_n}(0)\|_{\Cal E}\to\|\xi_{u}(0)\|_{\Cal E},\ n\to\infty
\end{equation}
%$$
which together with the weak convergence \eqref{3.weak1} implies the strong convergence
%$$
\begin{equation}
\xi_{u_n}(0)\to\xi_u(0) \ \text{ in the energy space }\ \ \Cal E
\end{equation}
%$$
and finishes the proof of the theorem.
\end{proof}
\begin{remark}\label{Rem4.bad} Remind that the asymptotic compactness of the solution semigroups associated with dissipative PDEs in unbounded domains are usually proved using the so-called weighted energy estimates or/and the so-called tail estimates, see \cite{Ab2,BV1,MZ,Z4} and references therein.
However, in contrast to many examples considered in \cite{MZ}, the hyperbolic CHO equation {\it does not} possess a weighted energy theory at least in the form used in \cite{MZ} and the weighted estimates require more regularity of the solutions. As we will see in the next section, the extra regularity of solutions require in turn the asymptotic compactness (at least under the approach used there). Thus, we do not see how to apply the weighted energy estimates for verifying the asymptotic compactness of the hyperbolic CHO equation and are forced  to use the energy method instead.
\par
Note also that despite the fact that the energy method is completely standard nowadays, the application of it to the hyperbolic CHO equation is a bit delicate due the presence of homogeneous Sobolev spaces for which we do not have the embedding $\dot H^{s_1}\subset\dot H^{s_2}$ for $s_1>s_2$. By this reason, we give the detailed exposition of the method in this section.
\end{remark}

\section{Smoothness of the global attractor}\label{s5}
The aim of this section is to verify that under the global attractor $\Cal A$ of the hyperbolic CHO equation is actually more smooth and, in particular, is a bounded set in the second energy space
%$$
\begin{equation}\label{4.space}
\Cal E_2:=[\dot H^{-1}\cap \dot H^3]\times[\dot H^{-1}\cap \dot H^1].
\end{equation}
%$$
To verify this fact, we will essentially use the asymptotic compactness proved in the previous section and its analogue for the space of trajectories.
\begin{proposition}\label{Prop4.kcomp} Let the assumptions of Theorem \ref{Th3.attr} hold and let $\Cal K$ be the set of all complete bounded Strichartz solutions of equation \eqref{2.CHO}. Then the set $\Cal K$ is compact in the space $C_{loc}(\R,\Cal E)\cap L^4_{loc}(\R,C(\R^3))$:
%$$
\begin{equation}\label{4.comp}
\Cal K\subset\subset C_{loc}(\R,\Cal E)\cap L^4_{loc}(\R,C(\R^3)).
\end{equation}
%$$
\end{proposition}
Indeed, the assertion of the proposition is an immediate corollary of the Lipschitz continuity of the solutions semigroup $S(t)$, see \eqref{2.lip} and the compactness of the global attractor $\Cal A$ in the phase space $\Cal E$.
\par
The following simple corollary of the proved compactness is however crucial for our method.
\begin{corollary}\label{Cor4.app1} Let the assumptions of Proposition \ref{Prop4.kcomp} hold. Then, for every $u\in\Cal K$ and every $\eb>0$, there exist functions $A_\eb(t)$ and $B_\eb(t)$ such that $u(t)=A_\eb(t)+B_\eb(t)$ and
%$$
\begin{equation}
\|\xi_{A_\eb}\|_{L^\infty(\R,\Cal E)}+\|A_\eb\|_{L^4(t,t+1;L^\infty)}\le \eb,\ \ t\in\R
\end{equation}
%$$
and $B_\eb\in C^1_b(\R, H^2(\R^3))$ with the estimate
%$$
\begin{equation}
\|B_\eb\|_{C^1_b(\R,\Cal E\cap H^2(\R^3))}\le C_\eb.
\end{equation}
%$$
Moreover, the constant $C_\eb$ is independent of $u\in\Cal K$.
\end{corollary}
Indeed, the assertion of the corollary follows in a straightforward way from the compactness of $\Cal K$, its invariance with respect to time shifts and the Hausdorff criterion. We just mention that the fact that $u\in L^4_{loc}(\R,C(\R^3))$ (and even $u\in L^4_{loc}(\R,C_0(\R^3))$ according to Remark \ref{Rem1-C0}, not only $L^4_{loc}(\R,L^\infty(\R^3))$) is crucial here since the smooth functions are not dense in $L^\infty(\R^3)$ and the approximation by smooth functions does not work in $L^\infty$.
\par
Analogously to this corollary, we also have the following result.
\begin{corollary}\label{Cor4.app2} Let the assumptions of Proposition \ref{Prop4.kcomp} hold. Then, for every $u\in\Cal K$ and every $\eb>0$, there exist functions $A^{f'}_\eb(t)$ and $B^{f'}_\eb(t)$ such that $f'(u(t))=A^{f'}_\eb(t)+B^{f'}_\eb(t)$ such that
%$$
\begin{equation}\label{4.afp}
\|A^{f'}_\eb\|_{L^1(t,t+1;L^\infty)}\le \eb,\ \ \|B^{f'}_\eb\|_{W^{1,\infty}(\R\times\R^3)}\le C_\eb,\ \ t\in\R,
\end{equation}
%$$
where the constant $C_\eb$ is independent of $u\in\Cal K$. Moreover, there exist functions $A^{f''}_\eb(t)$ and $B^{f''}_\eb(t)$ such that
$f''(u(t))=A^{f''}_\eb(t)+B^{f''}_\eb(t)$ and
%$$
\begin{equation}
\|A^{f''}_\eb\|_{L^{4/3}(t,t+1;L^\infty)}\le \eb,\ \ \|B^{f''}_\eb\|_{L^\infty(\R\times\R^3)}\le C_\eb,\ \ t\in\R,
\end{equation}
%$$
where $C_\eb$ is also uniform with respect to $u\in\Cal K$.
\end{corollary}
Indeed, since $f'$ and $f''$ are continuous, the sets $f'(\Cal K)$ and $f''(\Cal K)$ are also compact in the proper spaces and the approximation arguments work.
\par
We are now ready to state the main result of the section.
\begin{theorem}\label{Th4.smooth} Let the assumptions of Theorem \ref{Th3.attr} hold. Then the global attractor $\Cal A$ of the Strichartz solution
semigroup $S(t)$ associated with the hyperbolic CHO equation \eqref{2.CHO} is bounded in the second energy space $\Cal E_2$.
\end{theorem}
\begin{proof} We give here only the formal derivation of the control of the norm of $\Cal A$ in $\Cal E_2$ which can be justified by approximating the time derivative $v(t):=\Dt u(t)$ by finite differences. Indeed, let $u\in\Cal K$ be an arbitrary complete bounded trajectory. Then, the function $v(t)=\Dt u(t)$ solves
%$$
\begin{equation}\label{4.dif}
\Dt^2 v+\Dt v+\Dx(\Dx v-f'(u(t))v)+\alpha v=0.
\end{equation}
%$$
We complete the proof of the theorem in 3 steps. At Step 1, we estimate the energy norm of $v$ through its Strichartz norm. Then, at Step 2, we use the Stricharz estimate  to verify that
%$$
\begin{equation}\label{4.v}
\|\xi_v(t)\|_{\Cal E}+\|v\|_{L^4(t,t+1;L^\infty)}\le C,
\end{equation}
%$$
where the constant $C$ is independent of $t\in\R$ and $u\in\Cal K$. Finally, at Step 3, we use the elliptic regularity to verify that $u\in H^3(\R^3)$ which will complete the proof.
\par
{\it Step 1. Energy estimate.} We multiply equation \eqref{4.dif} by $\Dt\Dx^{-1}v+\delta\Dx^{-1}v$ and integrate over $x$. Then, arguing as in the derivation of \eqref{1.lin-big}, we end up with
%$$
\begin{equation}\label{4.mainen}
\frac d{dt}\Cal E_v(t)+\beta\Cal E_v(t)\le -\delta(f'(u(t)),v^2(t)) -(f'(u(t))v(t),\Dt v(t)).
\end{equation}
%$$
The first term on the right hand side of \eqref{4.mainen} can be easily estimated as follows
\begin{equation}
|\delta (f'(u),v^2)|\leq C(1+\|u\|^4_{L^{\infty}})\|v\|^2_{L^2}\le C(1+\|u\|^4_{L^{\infty}})[\E_v(t)]^\frac{1}{2},
\end{equation}
where at the last step we used $[\dot H^{-1},\dot H^1]_\frac12=L^2$ and uniform control of $\|\Dt u\|_{\dot H^{-1}}$ on $\Cal K$.

To estimate the second term on the right hand side of \eqref{4.mainen} we rewrite it in the form
%$$
\begin{multline}
-(f'(u)v,\Dt v)=(\Nx(f'(u)v),\Nx\Dx^{-1}\Dt v)=(f'(u)\Nx v,\Nx\Dx^{-1}\Dt v)+\\+(f''(u)\Nx u\, v,\Nx\Dx^{-1}\Dt v)=I_1+I_2.
\end{multline}
%$$
Using Corollary \ref{Cor4.app2}, we further transform the first term $I_1$ as follows
%$$
\begin{multline}
I_1=(A_\eb^{f'}(t)\Nx v,\Nx\Dx^{-1}v)+(B_\eb^{f'}(t)\Nx v,\Nx\Dx^{-1}\Dt v)\le\\
\le C\|A^{f'}_\eb(t)\|_{L^\infty}\Cal E_v(t)-(\Nx B_\eb^{f'}(t)v,\Nx\Dx^{-1}\Dt v)-(B_\eb^{f'}(t)v,\Dt v)\le\\
\le C\|A^{f'}_\eb(t)\|_{L^\infty}\Cal E_v(t)-\frac d{dt}(B_\eb^{f'}v,v)+C_\eb\|v\|_{L^2}\|\Dt v\|_{\dot H^{-1}}+C_\eb\|v\|^2_{L^2}\le\\
\le (C\|A^{f'}_\eb(t)\|_{L^\infty}+\eb)\Cal E_v(t)-\frac d{dt}(B_\eb^{f'}v,v)+C_\eb[\Cal E_v(t)]^{1/2},
\end{multline}
%$$
where we have implicitly used the interpolation
%$$
\begin{equation}
C_\eb\|v\|_{L^2}^2\le C_\eb\|v\|_{\dot H^1}\|\Dt u\|_{\dot H^{-1}}
\end{equation}
%$$
and the fact that the $\dot H^{-1}$-norm of $\Dt u$ is uniformly bounded on $\Cal K$.
\par
To estimate the second term $I_2$, we note that, due to Corollaries \ref{Cor4.app1} and \ref{Cor4.app2}, for every $\eb>0$, we have
%$$
\begin{equation}\label{4.split}
f''(u)\Nx u=K_\eb(t)+D_\eb(t)
\end{equation}
%$$
and
%$$
\begin{equation}\label{4.keb}
\|K_\eb\|_{L^{4/3}(t,t+1;L^2)}\le C\eb,\ \ \|D_\eb\|_{L^\infty(\R\times\R^3)}\le C_\eb,
\end{equation}
%$$
where the constant $C_\eb$ is uniform with respect to $u\in\Cal K$. Therefore,
%$$
\begin{multline}
I_2=(K_\eb(t) v,\Nx\Dx^{-1}\Dt v)+(D_\eb(t) v,\Nx\Dx^{-1}\Dt v)\le\\\le \|K_\eb(t)\|_{L^2}\|v(t)\|_{L^\infty}\|\Dt v\|_{\dot H^{-1}}+C_\eb\|v\|_{L^2}\|\Dt v\|_{\dot H^{-1}}\le\\\le \eb\Cal E_v(t)+(\|K_\eb(t)\|_{L^2}\|v(t)\|_{L^\infty}+C_\eb)[\Cal E_v(t)]^{1/2}.
\end{multline}
%$$
Inserting the obtained estimates in the right-hand side of \eqref{4.mainen} we end up with
%$$
\begin{multline}\label{4.mainen1}
\frac d{dt}\(\Cal E_v(t)+(B_\eb^{f'}(t)v,v)\)+(\beta-C\|A_\eb^{f'}(t)\|_{L^\infty}) \Cal E_v(t)\le\\\le (C_\eb+\|K_\eb(t)\|_{L^2}\|v(t)\|_{L^\infty}+\|u\|^4_{L^\infty})[\Cal E_v(t)]^{1/2}.
\end{multline}
%$$
We now fix $\bar{\Cal E}_v(t):=\Cal E_v(t)+(B_\eb^{f'}(t)v,v)+L\|v\|_{\dot H^{-1}}^2$, where the constant $L=L_\eb$ is such that
%$$
\begin{equation}
\frac14\|\xi_v(t)\|^2_{\Cal E}\le \bar{\Cal E}_v(t)\le C'_\eb\|\xi_v(t)\|^2_{\Cal E}
\end{equation}
%$$
(it is possible to find such $L$ and $C_\eb$ since $B_\eb(t)$ is bounded). Then, using the fact that the $\dot H^{-1}$ norm of $v$ is under the control, we transform  \eqref{4.mainen1} to
%$$
\begin{multline}\label{4.mainen2}
\frac d{dt}\bar{\Cal E}_v(t)+(\beta-C\|A_\eb^{f'}(t)\|_{L^\infty})\bar{\Cal E}_v(t)\le\\\le (C_\eb+C\|K_\eb(t)\|_{L^2}\|v\|_{L^\infty}+\|u\|^4_{L^\infty}+C_\eb\|A_\eb^{f'}(t)\|_{L^\infty})[\bar{\Cal E}_v(t)]^{1/2},
\end{multline}
%$$
where the positive constants $\beta$, $C$ and $C_\eb$ are independent of $t$ and $u\in\Cal K$. Applying the Gronwall inequality to \eqref{4.mainen2}
and using estimate \eqref{4.afp} to control the norm of $A_\eb^{f'}$, we infer
%$$
\begin{equation}
\|\xi_v(T)\|^\frac{1}{2}_{\Cal E}\le C[\bar{\Cal E}_v(0)]^{1/2}e^{-\beta T}+C_\eb+C\int_0^Te^{-\beta(T-s)}\|K_\eb(s)\|_{L^2}\|v(s)\|_{L^\infty}\,ds,
\end{equation}
%$$
where the constant $C$ is independent of $\eb$. Using now the control \eqref{4.keb} for estimating the integral in the right-hand side, we arrive at
%$$
\begin{equation}
\|\xi_v(T)\|_{\Cal E}\le C[\bar{\Cal E}_v(0)]^{1/2}e^{-\beta T}+C_\eb+C\eb\sup_{t\in\R}\|v\|_{L^4(t,t+1;L^\infty)}.
\end{equation}
%$$
Finally, passing to the limit $T\to\infty$ and using the invariance of $\Cal K$ with respect to time shifts, we infer that, for every $\eb>0$ there exists the constant $C_\eb$ such that, for every $u\in\Cal K$,
%$$
\begin{equation}\label{4.en-stri}
\|\xi_v(t)\|_{\Cal E}\le \eb\sup_{t\in\R}\|v\|_{L^4(t,t+1;L^\infty)}+C_\eb, \ \ t\in\R,
\end{equation}
%$$
where $C_\eb$ is independent of $u\in\Cal K$ and $t\in\R$.
\par
{\it Step 2. Strichartz estimate.}
We treat equation \eqref{4.dif} as a linear one with the external force  $f'(u)v$ and apply the dissipative estimate \eqref{1.lin-str} on the time interval $t\in[0,1]$:
%$$
\begin{multline}\label{4.stri}
\|v\|_{L^4(0,1;L^\infty)}\le C(\|\xi_v(0)\|_{\Cal E}+\|f'(u)v\|_{L^1(0,1;\dot H^1)})\le\\
\le C\eb\|v\|_{L^4_b(\R,L^\infty)}+C_\eb+C\|f'(u)v\|_{L^1(0,1;\dot H^1)},
\end{multline}
%$$
where $\|v\|_{L^4_b(\R,L^\infty)}:=\sup_{t\in\R}\|v\|_{L^4(t,t+1;L^\infty)}$.
Thus, we only need to estimate the last term in the right-hand side:
%$$
\begin{multline}
\|f'(u(t))v(t)\|_{L^1(0,1;\dot H^1)}\le \|f'(u(t))\Nx v(t)\|_{L^1(0,1;L^2)}+\\+\|f''(u(t))\Nx u(t)v(t)\|_{L^1(0,1;L^2)}:=J_1+J_2.
\end{multline}
%$$
Using the fact that the Strichartz norm of $u$ is uniformly bounded together with estimate \eqref{4.en-stri}, we may estimate the term $J_1$ as follows:
%$$
\begin{multline}
J_1\le \|f'(u)\|_{L^1(0,1;L^\infty)}\|v\|_{L^\infty(0,1;\dot H^1)}\le\\\le C\|\xi_v\|_{L^\infty(0,1;\Cal E)}\le C\eb\|v\|_{L^4_b(\R,L^\infty)}+C_\eb,
\end{multline}
%$$
where the constant $C$ is independent of $\eb$ and $u\in\Cal K$.
\par
To estimate the term $J_2$, we use again the splitting \eqref{4.split} with some new $\eb_1>0$ together with estimate \eqref{4.en-stri}:
%$$
\begin{multline}
J_2= \|K_{\eb_1} v\|_{L^1(0,1;L^2)}+\|D_{\eb_1}v\|_{L^1(0,1;L^2)}\le \\\le C\|K_{\eb_1}\|_{L^{4/3}(0,1;L^2)}\|v\|_{L^4(0,1;L^\infty)}+C_{\eb_1}\|v\|_{L^\infty(0,1;L^2)}\le \\\le C\eb_1\|v\|_{L^4_b(\R,L^\infty)}+C_{\eb_1}\|\xi_v\|_{L^\infty(0,1;\Cal E)}\le (C\eb_1+C_{\eb_1}\eb)\|v\|_{L^4_b(\R,L^\infty)}+C_{\eb}C_{\eb_1}.
\end{multline}
%$$
Inserting the obtained estimates into the right-hand side of \eqref{4.stri} and fixing first $\eb_1$ to be small enough and then $\eb\le\eb_1$ in such way that $C_{\eb_1}\eb\le \eb_1$, we end up with
%$$
\begin{equation}
\|v\|_{L^4(0,1;L^\infty)}\le C\eb_1\|v\|_{L^4_b(\R,L^\infty)}+C_{\eb_1},
\end{equation}
%$$
where the constant $C$ is independent of $\eb_1$ both constants $C$ and $C_{\eb_1}$ are independent of $u\in\Cal K$. Since $\Cal K$ is translation invariant, the last estimate (with sufficiently small $\eb_1$) implies the desired estimate for the Strichartz norm:
%$$
\begin{equation}
\|v\|_{L^4_b(\R,L^\infty)}\le C,\ \ v=\Dt u,\ \ u\in\Cal K
\end{equation}
%$$
which together with estimate \eqref{4.en-stri} gives the desired estimate \eqref{4.v} and completes the estimate of the energy norm of $\Dt u$.
\par
{\it Step 3. Elliptic regularity.}  Estimate \eqref{4.v} guarantees that $\Dt^2 u(t)=\Dt v(t)$ is uniformly bounded in $\dot H^{-1}$, so since $\alpha u(t)$ is also uniformly bounded in $\dot H^{-1}$ applying the operator $\Dx^{-1}$ to equation \eqref{2.CHO}, we see that
%$$
\begin{equation}\label{4.ell}
\Dx u(t)-f(u(t))-u(t)=-g-u(t)-\Dx^{-1}(\Dt^2 u(t)-\Dt u(t)-\alpha u(t)):=H_u(t)
\end{equation}
%$$
and the right-hand side $H_u(t)$ is uniformly bounded in $\dot H^1$. Moreover, using the growth restriction \eqref{2.f} on the function $f$ and the fact that $f(0)=0$, we see that $|f(u)|\le C(|u|+|u|^5)$ and, consequently, $f(u)$ is uniformly bounded in $H^{-1}(\R^3)$. Since $\Dx u(t)$ and $u(t)$ are also uniformly bounded in $H^{-1}(\R^3)$, then using the obvious embedding
%$$
\begin{equation}
 \dot H^1(\R^3)\cap H^{-1}(\R^3)\subset H^1(\R^3),
\end{equation}
%$$
we see that $H_u(t)\in H^1(\R^3)$ and
%$$
\begin{equation}
\|H_u(t)\|_{H^1}\le C,\ \ u\in\Cal K.
\end{equation}
%$$
Thus, we may use the classical maximal regularity in Sobolev spaces of the semilinear elliptic equation \eqref{4.ell} to establish that
$u\in H^3(\R^3)$ and
%$$
\begin{equation}
\|u(t)\|_{H^3(\R^3)}\le C, \ \ t\in\R,\ \ u\in \Cal K.
\end{equation}
%$$
This estimate, together with estimate \eqref{4.v} for the time derivative $v=\Dt u$ gives that $\xi_u(t)$ is uniformly bounded in $\Cal E_2$ and finishes the proof of the theorem.
\end{proof}
\begin{remark} Arguing analogously, one may show that the actual regularity of the global attractor $\Cal A$ is restricted by the regularity of the non-linearity $f$ and the external force $g$. In particular, if they are both $C^\infty$-smooth, the attractor will be $C^\infty$-smooth as well.
\end{remark}

\section{Generalizations and concluding remarks}\label{s6}
In this concluding section, we discuss possible generalizations of the result obtained above as well as the related open problems.

\begin{remark}\label{Rem5.stri} As we have seen, the global well-posedness of the hyperbolic CHO equation is crucially related with the validity of the Strichartz estimate \eqref{1.schro-stri} for the linear Schr\"odinger equation \eqref{1.schro}. In turn, this estimate follows from the standard Strichartz estimate
%$$
\begin{equation}\label{5.schro-good}
\|U\|_{L^p(-T,T;L^q(\R^3))}\le C\|U_0\|_{\dot H^s(\R^3)},\ \frac2p+\frac3q=\frac32-s
\end{equation}
%$$
for the homogeneous Scr\"odinger equation: $U(t):=e^{it\Dx}U_0$, see \cite{sogge1}. Thus, to apply the above technique for general domains $\Omega$, we need the analogue of estimate \eqref{5.schro-good} in $\Omega\subset\R^3$. However, to the best of our knowledge, the analogue of \eqref{5.schro-good} is known for the exterior domains $\Omega\subset\R^3$ only and for general bounded domain only the estimate
%$$
\begin{equation}
\|U\|_{L^p(-T,T;L^q(\Omega))}\le C\|U_0\|_{\dot H^{s+\frac1p}(\Omega)}
\end{equation}
%$$
with a loss of $\frac1p$ derivatives is known, see \cite{sogge1}. Thus, in contrast to the case of semilinear wave equations, see \cite{stri,plan2,KSZ2013} and references therein, the well-posedness result can be extended in a straightforward way to the case of exterior domains $\Omega$ only. The extension of the result to the case of hyperbolic CHO in bounded domains $\Omega$ is an interesting open problem which requires the additional non-trivial arguments to be involved. We refer the reader to the recent paper \cite{plan} where the global solvability of cubic non-linear Schr\"odinger equation has been established using the alternative method related with the so-called bilinear $L^2_{t,x}$ estimates, see also references therein.
\end{remark}
\begin{remark} Another interesting open problem is to extend the obtained result to the non-linearities $f$ of the critical quintic growth rate. In contrast to the subcritical case, we do not have here the control \eqref{2.str-en} of the Strichartz norm through the energy norm of the solution, so the global solvability becomes much more delicate and the analogue of the so-called non-concentration estimates (see \cite{noncon,sogge-book,straus,tao}) should be verified for the case of the hyperbolic CHO equations in order to get the desired global existence. Note also that these non-concentration estimates are the main obstacles for developing the attractor theory for the quintic growth rate. Indeed, if the global existence is known, one can use the so-called  backward regularity on the weak attractor in order to verify the asymptotic compactness analogously to the quintic wave equation, see \cite{KSZ2013}.
\end{remark}
\begin{remark} Finally, it worth mentioning that we did not use the dissipation integral for the proof of main results, so  these results can be easily extended to the case of non-autonomous external forces $g(t)$, e.g., in the sense of uniform or pullback attractors. However, the existence of {\it exponential} attractors is more delicate since, analogously to the Scr\"odinger equation, the hyperbolic CHO equation {\it does not possess} a reasonable weighted energy theory (at least similar to the one, developed in \cite{BV1,MZ,Z4}, see also the references therein). Thus, the standard methods of establishing the finite-dimensionality of a global attractor seem not working. Nevertheless, some analogues of weighted estimates can be proved if the additional regularity of the solution of difference between two solutions is known and based on this observation, one can establish the finite-dimensionality of a global attractor and the existence of exponential attractors, see the forthcoming paper \cite{SZ} for more details.
\end{remark}

\end{document}